\documentclass[12pt]{amsart}

\usepackage{datetime}
\usepackage{breakcites}

\usepackage{amsmath,amscd,amsfonts,amssymb}
\usepackage{fancyhdr}
                                 
\usepackage{graphicx}
\usepackage{tcolorbox} 
\usepackage{lineno}
\setpagewiselinenumbers
\usepackage{enumerate}
\usepackage[all]{xy}
\usepackage[colorlinks=true,urlcolor=red,citecolor=black,linkcolor=blue]
           {hyperref} 

\definecolor{mygreen}{rgb}{0.0, 0.5, 0.0} 
\newtheorem{thm}{Theorem} 
\newtheorem{lem}[thm]{Lemma}
\newtheorem{prop}{Proposition}
\newtheorem{defn}[thm]{Definition}

\theoremstyle{remark}

\newcommand{\defi}{:=}

\newcommand{\F}{{\mathbb F}}             

\newcommand{\Z}{{\mathbb Z}}

\newcommand{\tor}{\mathop\mathrm{tor}}

\newcommand{\id}{\mathop{\rm id}\nolimits}

\newcommand{\msk}{\medskip}

\newcommand{\ignore}[1]{} 

\definecolor{mygreen}{rgb}{0.0, 0.5, 0.0} 

\long\def\alert#1{\parindent2em\smallskip\hbox to\hsize%
{\hskip\parindent\vrule%
\vbox{\advance\hsize-2\parindent\hrule\smallskip\parindent.4\parindent%
\narrower\noindent#1\smallskip\hrule}\vrule\hfill}\smallskip\parindent0pt}

\newcounter{alert}0

\newcommand{\lc}{lo\-cal\-ly com\-pact}
\newcommand{\lca}{\lc\ a\-bel\-ian}

\newcommand{\lead}{\leaders\hbox to 1.5ex{\hss${.}$\hss}\hfill}
\newcommand{\arr}{\hbox to 20pt{\rightarrowfill}}
\newcommand{\larr}{\hbox to 20pt{\leftarrowfill}}

\newcommand{\refeq}[1]{Eq.\,(\ref{eq:#1})}

\makeindex

\newtheorem{exam}[thm]{Example}

\begin{document}
\title[A Short Note on Coproducts of Abelian pro-Lie Groups]{A Short Note on Coproducts of Abelian pro-Lie Groups}
\author[W. Herfort]{Wolfgang Herfort}
\address{\newline  
Institute for Analysis and Scientific Computation  \newline
Technische Universit\"{a}t Wien\newline 
Wiedner Hauptstra\ss e 8-10/101\newline                    
Vienna, Austria} 
\email{wolfgang.herfort@tuwien.ac.at}

\author[K. H. Hofmann]{Karl H. Hofmann}
\address{\newline 
Fachbereich Mathematik\newline 
Technische Universit\"at Darmstadt\newline 
Schlossgartenstr. 7\newline e
64289 Darmstadt, Germany} 
\email{hofmann@mathematik.tu-darmstadt.de}

\author[F. G. Russo]{Francesco G. Russo}
\address{\newline
Department of Mathematics and Applied Mathematics \newline 
University of Cape Town\newline
Private Bag X1, Rondebosch 7701\newline
Cape Town, South Africa}
\email{francescog.russo@yahoo.com}
\keywords{Pro-Lie groups, coproducts\endgraf
\textit{Mathematics Subject Classification 2020:} 22E20; 22A05}
\date{\today}
\begin{abstract}
The notion of {\em conditional coproduct} of a family of abelian pro-Lie groups
 in the category of abelian pro-Lie groups is introduced. 
It is shown that the Cartesian product of an arbitrary family of 
abelian pro-Lie groups can be characterized by the universal property of the
{\em conditional coproduct}.
\end{abstract}
\maketitle  
In \cite{hofmor77} the second author and S.~Morris have provided
criteria when the (co)product of a family of  \lca\ groups exists.
For profinite abelian groups the product of any family
$(A_i)_{i\in I}$
of profinite
groups exists and agrees with the cartesian product
$P:=\prod_{i\in I}A_i$.
J.~Neukirch  has shown in \cite{Neukirch71} that $P$
has a universal property resembling that of 
 a coproduct (direct sum) in the category of (discrete) abelian groups.
In the present note we present a version of his result, valid for
cartesian products of 
the much larger
family of {\em abelian pro-Lie groups} (see \cite[Ch. 5]{hofmor-conn}).
For formulating our result, we 
need to adapt the concepts, originally introduced
for families of profinite groups by
J. Neukirch in \cite{Neukirch71} (see also \cite[D.3]{ribes-zalesskii}),
to the category of abelian pro-Lie groups.

\begin{defn}\label{d:convergent}\rm
Let $(A_j)_{\in J}$ be a family of topological
groups, $H$ a topological group, and
$\mathbb F=(\phi_j)_{j\in J}$,  $\phi_j\colon A_j\to H$,
a family of continuous homomorphisms. We say that
$\mathbb F$ is {\em convergent}, if for every identity neighborhood $U$
of $H$ the set $J_U\defi \{j\in J: \phi_j(A_j)\not\subseteq U\}$
is finite.
\end{defn}

\begin{exam}\label{ex:P-convergent}\rm 
For any family $(A_j)_{j\in J}$ of
topological groups let $H=\prod_{j\in J}A_j$ be the
cartesian product with the Tychonov topology.
Let the family $\F=(\tau_j)_{j\in J}$ of natural morphisms
$\tau_j\colon A_j\to H$ be given by
\[\tau_j(a)=(a_k)_{k\in J}\mbox{ with }a_j=a\mbox{ and }a_k=0\mbox{  otherwise.}\]
Then $\F$ {\em is convergent}.

\medskip

\noindent
This follows immediately from the definition of the product topology
on $H$. The morphisms $\tau_j$ are called the {\em natural embeddings}.

\end{exam}

We define the {\em conditional coproduct} by means of
a universal property,
resembling the one of the coproduct (direct sum) of abelian discrete groups:

\begin{defn}\label{d:condprod}\rm 
In a category ${\mathcal A}$ of topological groups we call
$G$ a {\em conditional coproduct} of the family $(A_j)_{j\in J}$
of objects if there is a convergent family
$\tau_j\colon A_j\to G$, $j\in J$ of morphisms such that for every
convergent family of morphisms $\psi_j\colon A_j\to H$, $j\in J$
in $\mathcal A$ there is a unique morphism $\omega\colon G\to H$
such that $\psi_j=\omega\circ\tau_j$ for all $j\in J$.
The morphisms $\tau_j$  are called the {\em coprojections}
of the conditional coproduct.
\end{defn}

We shall prove the following Theorem:
  
\begin{thm}\label{t:tychonov}
In the category of  abelian pro-Lie groups, the
conditional coproduct of a family $(A_j)_{j\in J}$ of abelian
pro-Lie groups is the cartesian product
$P\defi\prod_{j\in J} A_j$
for
the canonical embeddings $\tau_j\colon A_j\to P$.
\end{thm}




But first we secure the uniqueness of the
conditional coproduct:

\begin{prop}\label{p:condprodunique} If $G$ and $G'$ are conditional
coproducts of a family $(A_j)_{j\in J}$
of topological groups in a category $\mathcal A$
for the convergent families  $\tau_j\colon A_j\to G$
and $\tau'_j\colon A_j\to G'$, $j\in J$ of morphisms in
$\mathcal A$, then there is a natural isomorphism
$\lambda\colon G\to G'$ such that
$\tau'_j=\lambda\circ \tau_j$ for all $j\in J$.
\end{prop}

\begin{proof} By Definition 2,
  since $G$ is a conditional coproduct of the family $(A_j)_{j\in J}$
  with  the coprojections $\tau_j$, $j\in J$, there is a unique morphism
  $\lambda \colon G\to G'$ such that
  \begin{equation}\label{eq:1}
\forall j\in J)\, \lambda\circ \tau_j =\tau'_j\colon A_j\to G'.
  \end{equation}
  Likewise, since $G'$ is also a conditional coproduct
  of the family $(A_j)_{j\in J}$ with the coprojections $\tau'_j$,
  $j\in J$, there is a unique morphism $\lambda'\colon G'\to G$
  such that
  \begin{equation}\label{eq:2}
(\forall j\in J)\, \lambda'\circ \tau'_j= \tau_j\colon A_j\to G.
  \end{equation}
  Therefore, by (\ref{eq:1}) and (\ref{eq:2}),  we have
  \begin{equation}\label{eq:3}(\forall j\in J)\, \tau_j=\lambda'\circ \tau'_j
  =\lambda'\circ \lambda \circ \tau_j\colon G\to G. \end{equation}
  However, trivially we also have,
  \begin{equation}\label{eq:4}
(\forall j\in J)\, \tau_j=\id_G\circ \tau_j\colon G\to G.
\end{equation}
  Therefore, by the uniqueness in Definition \ref{d:condprod}, 
from (\ref{eq:3})  and (\ref{eq:4}) we have 
\begin{equation}\label{eq:5}
\lambda'\circ \lambda=\id_G.
\end{equation}
  Now by exchanging the roles
  of $G$ and $G'$ we also have
\begin{equation}\label{eq:6}
\lambda\circ\lambda'=\id_{G'}.
\end{equation}
  Hence by (\ref{eq:5}) and (\ref{eq:6}), $\lambda$ is an isomorphism,
  which we had to show.
   \end{proof}

We note that for profinite groups the {\em conditional coproduct}
agrees with the {\em free pro-$\mathcal C$ product} for $\mathcal C$
the variety of abelian profinite groups, see \cite{Neukirch71,ribes-zalesskii}.

\msk

Let $\mathcal A$ be the category of topological abelian pro-Lie groups
(i.e. groups which are projective limits of Lie groups: see
\cite[pp. 160ff. and Chapter 5]{hofmor-conn}).
Each pro-Lie group $G$ has a filterbasis
${\mathcal N}(G)$ of closed normal subgroups such that
$G/N$ is a Lie group, and $G\cong\lim_{N\in{\mathcal N}(G)}G/N$.
(See e.g. \cite[p.160, Definition A.]{hofmor-conn})

Recall that
every locally compact abelian group is a pro-Lie group, every
almost connected locally compact group is a pro-Lie group 
by Yamabe's Theorem.  Trivially, then, every profinite group is a pro-Lie group.
Every cartesian product $P=\prod_{j\in J}A_j$ of pro-Lie groups
 $A_j$ is itself a pro-Lie group.

\begin{lem} \label{l:JN-finite}
Let $H$ be an abelian pro-Lie group and  $\F$ be  a convergent family
of morphisms $\psi_j:A_j\to H$. 
Then, for each $N\in {\mathcal N}(H)$, 
the set $\{j\in J:\psi_j(A_j)\not\subseteq N\}$ is finite.
\end{lem}

\begin{proof} Let $N\in {\mathcal N}(H)$.
The Lie group $H/N$ has an identity neighborhood
$V$ in which $\{0\}$ is the only subgroup of $H/N$. Now
let $p\colon H\to H/N$ be the quotient morphism and set
$U=p^{-1}(V)$.  

Therefore $\psi_j(A_j)\not\subseteq N$ implies
$\psi_j(A_j)\not\subseteq U$. However the set of
$j$ satisfying this condition is finite by Definition \ref{d:convergent}
applying to the conditional coproduct of the family
$(A_j)_{j\in J}$. This completes the proof of the Lemma.
\end{proof}

\begin{proof}[Proof of Theorem \ref{t:tychonov}]
The uniqueness, up to isomorphism, of the conditional coproduct
follows from Proposition \ref{p:condprodunique}.

Thus, according to Definition \ref{d:condprod}, we need to show 
that given an abelian pro-Lie group $H$ and a convergent family of morphisms
$\psi_j:A_j\to H$ then there exists a unique morphism 
$\omega:P\to H$ with $\psi_j=\omega\circ\tau_j$
for all $j\in J$.

\msk

We note first that every $x\in P$ has a presentation 
\begin{equation}\label{eq:xinP}
x=(\tau_j(a_j))_{j\in J}\end{equation} 
for unique elements $a_j\in A_j$. 
Denote by $\mathcal N(H)$ the set of all closed subgroups of $H$ such
that $H/N$ is a Lie group. It is a consequence 
of \cite[Theorem 3.27]{hofmor-conn} that $\mathcal N(H)$ 
is a filter basis of closed subgroups of $H$ and
that  
\begin{equation}\label{eq:H-lim}
H\cong \varprojlim_{N\in\mathcal N}H/N
\end{equation}
algebraically and topologically. 

Fix
$N\in \mathcal N(H)$ and let let $J_N\defi\{j\in J:\psi_j(A_j)\not\le N\}$.
Then, by Lemma \ref{l:JN-finite}, the set  $J_N$ is finite and, 
taking the presentation \refeq{xinP} for
$x\in P$ and $\tau_j(A_j)\le N$ for all $j\notin J_N$
into account, we obtain a well-defined morphism 
$\omega_N:P\to H/N$ by letting
\begin{equation}\label{eq:omegaN}
\omega_N(x)\defi \sum_{j\in J_N}\psi_j(a_j)+N.\end{equation}
For subgroups  $M\le N$ of $H$, both in $\mathcal N(H)$,
 let $\pi_{NM}:H/M\to H/N$ denote
the canonical epimorphism.  
 
For $M\le N$ one obtains from \refeq{omegaN} the compatibility relation  
\begin{equation}\label{eq:omegaMN}
\omega_N=\pi_{NM}\circ\omega_M, \end{equation}
as depicted in the following diagram:
\[ \xymatrix{&&&&P\ar@{->}[dlll]_{\omega_N}\ar@{->}[dl]^{\omega_M}\ar@{->}[d]^{\exists!\,\omega}\\
      \cdots  & H/N\ar@{<-}[rr]_{\pi_{NM}} && H/M \ar@{<-}[r] \ldots  & H=\varprojlim_{N\in\mathcal U}H/N
} 
\]
Taking the relations in \refeq{omegaMN} into account we see that
the universal property of the inverse limit 
$H=\varprojlim_{N\in\mathcal U}H/N$ 
implies the existence of a unique continuous 
homomorphism $\omega:P\to H$, which
satisfies the desired relations 
\begin{equation}\label{eq:omegapsi}
(\forall j\in J)\ \ \ \psi_j=\omega\circ\tau_j.
\end{equation}
\vskip-25pt
\end{proof}

\bigskip

\noindent{\bf Notes.}

\medskip

\noindent A {\em coproduct} of a family
of objects in a
category $\mathcal A$ is a {\em product} in  the category obtained
by reversing all arrows. Curiously, while products are
usually considered  simple concepts, coproducts are often
tricky in many categories $\mathcal A$ other than the category
of abelian groups. Therefore, in conclusion of this note,
a few general comments may be in order. 

One of the early surprises is that in the familiar
{\em category of groups}, the coproduct
of $\Z(2)$ and $\Z(3)$ is PSL$(2,\Z)$.

In any category $\mathcal A$ with
a well-introduced dual category, such as {\em the category of
locally compact abelian groups}, the coproduct
$\coprod_j A_j$ of a family $A_j$, $j\in J$,
is naturally isomorphic to
the dual $\widehat P$ of $P:=\prod_j \widehat{A_j}$,
the product of its duals.

Even in special cases,
such as the case of {\em compact} abelian groups $A_j$,
the result is a complicated coproduct, since the character
group of an infinite  product of discrete abelian groups may
be hard to deal with.

If $\mathcal A$ is {\em the category of profinite
abelian groups}, then its dual is the category $\mathcal T$ of abelian
torsion groups. The product in $\mathcal T$ of a family of torsion
groups $T_j$ is the torsion group $\tor(\prod_j T_j)$ of the
cartesian product.
So by the time we arrive at the coproduct of, say,  an 
unbounded family of cyclic groups $A_j$ in $\mathcal A$, we may have
a complicated object $\coprod_{j\in J} A_j$ in our hands.

Therefore, any special situation  may be welcome,
where a 
coproduct is lucid--even when its scope of application may be
restricted. An example
of such a situation is our present {\em conditional coproduct} in
the rather large yet reasonably well-understood category of
abelian pro-Lie groups (see Chapter 5 of \cite{hofmor-conn}).
The authors encountered such a coproduct  in a study of
certain locally compact abelian
$p$-groups. 
Our conditional coproduct covers a somewhat 
restricted supply of families of morphisms which we call
``convergent''. Here we encounter  the rather
extraordinary event that for each of such  families {\em their
conditional coproduct agrees with their cartesian product}.
Classically, one is familiar with  a situation of coproducts agreeing
with products in the category of finite abelian groups which, after all,
is rather special.

\bibliographystyle{alpha}
\def\cprime{$'$}
\def\cydot{\leavevmode\raise.4ex\hbox{.}}
\bibliography{../BOOK_ERRATA/BOOK}
\end{document}